\documentclass[11pt]{article}
\usepackage{latexsym}
\usepackage{amsfonts}
\usepackage{mathrsfs,amsthm,amsmath}
\usepackage{color}
\def\sgn#1{\mathrm{sgn}(#1)}

\newtheorem{theorem}{Theorem}[section]

\newtheorem{claim}[theorem]{Claim}

\newtheorem{lemma}[theorem]{Lemma}

\newtheorem{proposition}[theorem]{Proposition}
\newtheorem{remark}{Remark}[section]

\begin{document}
\title{Large deviations for a class of tempered subordinators 
and their inverse processes\footnote{CM and BP acknowledge
the support of GNAMPA-INdAM and of MIUR Excellence Department 
Project awarded to the Department of Mathematics, University of 
Rome Tor Vergata (CUP E83C18000100006).}}
\author{Nikolai Leonenko\thanks{Address: Cardiff School of 
Mathematics, Cardiff University, Senghennydd Road, Cardiff, CF24 
4AG, UK. e-mail: \texttt{leonenkon@cardiff.ac.uk}} \and
Claudio Macci\thanks{Address: Dipartimento di Matematica,
Universit\`a di Roma Tor Vergata, Via della Ricerca Scientifica,
I-00133 Rome, Italy. e-mail: \texttt{macci@mat.uniroma2.it}} \and
Barbara Pacchiarotti\thanks{Address: Dipartimento di Matematica,
Universit\`a di Roma Tor Vergata, Via della Ricerca Scientifica,
I-00133 Rome, Italy. e-mail: \texttt{pacchiar@mat.uniroma2.it}}}
\date{}
\maketitle
\begin{abstract}
We consider a class of tempered subordinators, namely a class of
subordinators with one-dimensional marginal tempered distributions
which belong to a family studied in \cite{BCCP}. The main 
contribution in this paper is a non-central moderate deviations
result. More precisely we mean a class of large deviation principles
that fill the gap between the (trivial) weak convergence of some
non-Gaussian identically distributed random variables to their common
law, and the convergence of some other related random variables to a 
constant. Some other minor results concern large deviations for the 
inverse of the tempered subordinators considered in this paper; 
actually, in some results, these inverse processes appear as random 
time-changes of other independent processes.\\
\ \\
\textbf{Keywords:} Mittag-Leffler function, non-central moderate 
deviations, random time-changes, Tweedie distribution.\\
\emph{2010 Mathematical Subject Classification}: 60F10, 60G52, 60J25.
\end{abstract}

\section{Introduction}
Several non-standard stochastic processes in the literature are defined by
$\{X(T(t)):t\geq 0\}$, where $\{T(t):t\geq 0\}$ is an independent random 
time-change of a standard stochastic process $\{X(t):t\geq 0\}$. An important
class of random time-changes is given by subordinators, i.e. nondecreasing 
L\'evy processes (see e.g. \cite{Bertoin} and \cite{Sato} as references on 
these processes); however, in several recent references the process 
$\{T(t):t\geq 0\}$ is the inverse of a subordinator.

The family of (positive) stable subordinators is widely studied. An important 
feature of these processes is that their finite dimensional distributions do 
not have finite moments. In some situations this could be a problem and this
explains the increasing popularity of the tempered version of stable 
subordinators. In fact these tempered processes have finite dimensional 
distributions with finite moments, and they keep some other properties of the 
stable subordinators themselves (for instance in both cases the finite dimensional
distributions are self-decomposable, and therefore infinite divisible). Here we 
recall \cite{Rosinski}, \cite{GajdaWylomanska}, \cite{Wylomanska2012}, 
\cite{Wylomanska2013}, \cite{KuchlerTappe}, \cite{KumarVellaisamy} as references 
on tempered stable processes, tempered stable subordinators and, in some cases, 
on inverse of stable subordinators; other more recent references are 
\cite{GajdaKumarWylomanska}, \cite{KumarUpadhyeWylomanskaGajda}, 
\cite{KumarGajdaWylomanskaPoloczanski}, \cite{GuptaKumarLeonenko} and 
\cite{KumarLeonenkoPichler}. We also recall \cite{Grabchak} as a brief survey on
tempered stable distributions and their associated L\'evy processes.

The interest of the processes studied in this paper is motivated by their connections
with important research fields as, for instance, the theory of fractional differential
equations (see e.g. \cite{KilbasSrivastavaTrujillo}; see also \cite{Beghin} for the 
tempered case) and the theory of processes with long-range dependence (see e.g. 
\cite{Samorodnitsky}).

In this paper we consider a 4-parameter family of infinitely divisible 
distributions introduced in \cite{BCCP} (Section 3), which is inspired by some 
ideas in \cite{KlebanovSlamova}. This family, which generalizes the Tweedie 
distribution (case $\delta=0$) and the positive Linnik distribution (case 
$\theta=0$), is constructed by considering the randomization of the parameter 
$\lambda$ with a Gamma distributed random variable. Actually, in our results, we
often have to restrict the analysis on the case $\delta=0$.

The asymptotic results presented in this paper concern the theory of large
deviations; see e.g. \cite{DemboZeitouni} as a reference on this topic. This 
theory gives asymptotic computations of small probabilities on an exponential 
scale. Here we also recall \cite{GulinskyVeretennikov} as a reference on large
deviations and the averaging theory.

We start with some preliminaries in Section \ref{sec:preliminaries}. Section 
\ref{sec:MD} is devoted to the main contribution in this paper, i.e. a class
of large deviation principles that can be seen as a result of 
\emph{non-central moderate deviations} with respect to $\theta$ (see Proposition
\ref{prop:MD-theta-to-infinity} as $\theta\to\infty$; see also Remark 
\ref{rem:case-theta-to-zero} for the case $\theta\to 0$). These large deviation 
principles fill the gap between two asymptotic regimes:
\begin{itemize}
\item a weak convergence to a non-Gaussian distribution; actually we have a 
family of identically distributed random variables, which converge weakly to 
their common law;
\item the convergence to a constant of some other related random variables.
\end{itemize}
This is illustrated in detail in Remark \ref{rem:MD-typical-features}. Obviously,
we use the term \lq\lq non-central\rq\rq\ because the weak limit in the first 
item is not Gaussian.

In Section \ref{sec:LD-inverse-processes} we present some other minor 
large deviation results for the inverse of the subordinators studied in 
this paper; this will be done by applying the results in \cite{DuffieldWhitt}. 
Some other minor results are presented in Section \ref{sec:LD-time-changes}, where
the inverse of the subordinators studied in this paper are random time-changes of 
other independent processes.

\section{Preliminaries and some remarks}\label{sec:preliminaries}
In this section we present some preliminaries on large deviations 
and on the family of tempered distributions introduced in \cite{BCCP}.

\subsection{Preliminaries on large deviations}\label{sec:preliminaries-LD}
Here we recall some preliminaries on the theory of large deviations; see e.g.
the definitions in \cite{DemboZeitouni}, pages 4-5. Let $\mathcal{Y}$ be a 
topological space, and let $\{Y_r\}_r$ be a family of $\mathcal{Y}$-valued 
random variables defined on the same probability space $(\Omega,\mathcal{F},P)$;
then $\{Y_r\}_r$ satisfies the large deviation principle (LDP from now on), as 
$r\to r_0$ (possibly $r_0=\infty$), with speed $v_r$ and rate function $I$ if: 
$v_r\to\infty$ as $r\to r_0$, $I:\mathcal{Y}\to [0,\infty]$ is a lower 
semicontinuous function, and the inequalities
$$\liminf_{r\to r_0}\frac{1}{v_r}\log P(Y_r\in O)\geq-\inf_{y\in O}I(y)\ \mbox{for all open sets}\ O$$
and
$$\limsup_{r\to r_0}\frac{1}{v_r}\log P(Y_r\in C)\leq-\inf_{y\in C}I(y)\ \mbox{for all closed sets}\ C$$
hold. A rate function is said to be good if  
$\{\{y\in\mathcal{Y}:I(y)\leq\eta\}:\eta\geq 0\}$ is a family of compact sets.

We essentially deal with cases where $\mathcal{Y}=\mathbb{R}^h$ for some integer
$h\geq 1$, and we often use the G\"artner Ellis Theorem (see e.g. Theorem 2.3.6
in \cite{DemboZeitouni}). Here we briefly recall the statement of that theorem
and, in view of what follows, throughout this paper we use the notation 
$\langle\cdot,\cdot\rangle$ for the inner product in $\mathbb{R}^h$. Assume that
there exists
$$\lim_{r\to r_0}\frac{1}{v_r}\log\mathbb{E}[e^{v_r\langle y,Y_r\rangle}]=\Lambda(y)\ (\mbox{for all}\ y\in\mathbb{R}^h)$$
as an extended real number; moreover assume that the function $\Lambda$ is finite
in a neighborhood of the origin (i.e. $y=0$, where $0\in\mathbb{R}^h$ is the null
vector), and it is lower semicontinuous and essentially smooth according to 
Definition 2.3.5 in \cite{DemboZeitouni}. Then $\{Y_r\}_r$ satisfies the LDP with
speed $v_r$ and good rate function $\Lambda^*$ defined by
$$\Lambda^*(x):=\sup_{y\in\mathbb{R}^h}\{\langle x,y\rangle-\Lambda(y)\}.$$
The function $\Lambda^*$ is called Fenchel-Legendre transform of the function $\Lambda$.

\begin{remark}\label{rem:convergence-under-GET}
Let us consider the above setting of the G\"artner Ellis Theorem and, for simplicity, we 
consider the case $r_0=\infty$. Moreover we consider the closed set
$C_\delta:=\{x\in\mathbb{R}:\|x-\nabla\Lambda(0)\|\geq\delta\}$ for some $\delta>0$;
then, since $\Lambda^*(x)=0$ if and only if $x=\nabla\Lambda(0)$, we have
$\Lambda^*(C_\delta):=\inf_{y\in C_\delta}\Lambda^*(y)>0$. We want to consider the LDP
upper bound for the closed set $C_\delta$. Then, for all $\varepsilon>0$ small enough,
there exists $r_\varepsilon$ such that
$$P(|Y_r-\nabla\Lambda(0)|\geq\delta)\leq e^{-v_r(\Lambda^*(C_\delta)-\varepsilon)}
\ \mbox{for all}\ r>r_\varepsilon.$$
Thus $Y_r$ converges to $\nabla\Lambda(0)$ in probability. Moreover it is possible
to check the almost sure convergence along a sequence $\{r_n:n\geq 1\}$ such that
$r_n\to\infty$; in fact, by a standard application of Borel Cantelli Lemma, we can
say that $Y_{r_n}$ converges to $\nabla\Lambda(0)$ almost surely if
\begin{equation}\label{eq:BC-lemma-convergence}
\sum_{n\geq 1}e^{-v_{r_n}(\Lambda^*(C_\delta)-\varepsilon)}<\infty;
\end{equation}
for instance, when $v_r=r$ (we have this situation in Sections 
\ref{sec:LD-inverse-processes} and \ref{sec:LD-time-changes}), condition
\eqref{eq:BC-lemma-convergence} holds with the sequence $r_n=n$.
\end{remark}

Here we also recall the contraction principle (see e.g. Theorem 4.2.1 in 
\cite{DemboZeitouni}), that will be used in Remark 
\ref{rem:increments-vs-marginals}. Let $\{Y_r\}_r$ be a family of 
$\mathcal{Y}$-valued random variables defined on the same probability (as above),
and assume that $\{Y_r\}_r$ satisfies the LDP, as $r\to r_0$, with speed $v_r$ 
and good rate function $I$. Then, if we consider a continuous function
$f:\mathcal{Y}\to\mathcal{Z}$, where $\mathcal{Z}$ is another topological space,
the family of $\mathcal{Z}$-valued random variables $\{f(Y_r)\}_r$ satisfies the
LDP, as $r\to r_0$, with speed $v_r$ and good rate function $J$ defined by
$$J(z):=\inf\{I(y):y\in\mathcal{Y},\ f(y)=z\}.$$

\subsection{Preliminaries on the tempered distributions in this paper}\label{sec:preliminaries-BCCP}
We consider a family of subordinators $\{S_{(\gamma,\lambda,\theta,\delta)}(t):t\geq 0\}$,
where the parameters $(\gamma,\lambda,\theta,\delta)$ belong to a suitable set
$\mathcal{P}:=\mathcal{P}_1\cup\mathcal{P}_2$, i.e.
$$\mathcal{P}_1=(-\infty,0)\times(0,\infty)\times(0,\infty)\times[0,\infty)\
\mbox{and}\ \mathcal{P}_2=(0,1)\times(0,\infty)\times [0,\infty)\times[0,\infty);$$
actually other cases could be allowed ($\gamma=0$ when 
$(\gamma,\lambda,\theta,\delta)\in\mathcal{P}_1$
and $\gamma=1$ when $(\gamma,\lambda,\theta,\delta)\in\mathcal{P}_2$) but they will
be neglected because they give rise to deterministic random variables. Then, for each 
$(\gamma,\lambda,\theta,\delta)\in\mathcal{P}$ and for all $t\geq 0$, we consider the
moment generating function
$$\mathbb{E}[e^{yS_{(\gamma,\lambda,\theta,\delta)}(t)}]=\exp(t\kappa_{(\gamma,\lambda,\theta,\delta)}(y))\ (\mbox{for all}\ y\in\mathbb{R}),$$
where
$$\kappa_{(\gamma,\lambda,\theta,\delta)}(y):=\log\mathbb{E}[e^{yS_{(\gamma,\lambda,\theta,\delta)}(1)}].$$
We remark that we are setting $y=-s$, where $s>0$ is the argument of the Laplace transforms 
in \cite{BCCP} and, moreover, we have $\mathbb{E}[e^{yS_{(\gamma,\lambda,\theta,\delta)}(1)}]=\infty$
for some $y>0$. Furthermore, in view of the applications of the G\"artner Ellis Theorem,
it is useful to introduce the Fenchel-Legendre transform of the function 
$\kappa_{(\gamma,\lambda,\theta,\delta)}$, i.e. the function 
$\kappa_{(\gamma,\lambda,\theta,\delta)}^*$ defined by
\begin{equation}\label{eq:Legendre-transform-kappa}
\kappa_{(\gamma,\lambda,\theta,\delta)}^*(x):=\sup_{y\in\mathbb{R}}\{xy-\kappa_{(\gamma,\lambda,\theta,\delta)}(y)\}.
\end{equation}
We remark that, when we deal with $\{S_{(\gamma,\lambda,\theta,\delta)}(t):t\geq 0\}$,
the G\"artner Ellis Theorem can be applied only when $\theta>0$; in fact in this case the 
function $\kappa_{(\gamma,\lambda,\theta,\delta)}$ is finite in a neighborhood of the 
origin $y=0\in\mathbb{R}$.

\paragraph{Case $\delta=0$.}
This is the case of Tweedie distribution (see Section 2.2 in \cite{BCCP} and the 
references cited therein). We have
$$\kappa_{(\gamma,\lambda,\theta,0)}(y):=\log\mathbb{E}[e^{yS_{(\gamma,\lambda,\theta,0)}(1)}]
=\lambda\hspace{1.0pt}\sgn{\gamma}(\theta^\gamma-(\theta-y)^\gamma)$$
if $y\leq\theta$, and equal to infinity otherwise. Note that
$$\kappa_{(\gamma,\lambda,\theta,0)}(y)=\lambda\kappa_{(\gamma,1,\theta,0)}(y).$$
Thus we have the two following cases:
$$\mbox{if}\ \gamma\in(-\infty,0),\quad\kappa_{(\gamma,\lambda,\theta,0)}(y):=\log\mathbb{E}[e^{yS_{(\gamma,\lambda,\theta,0)}(1)}]
=\left\{\begin{array}{ll}
\frac{\lambda}{\theta^{-\gamma}}\left(\left(\frac{\theta}{\theta-y}\right)^{-\gamma}-1\right)&\ \mbox{if}\ y<\theta\\
\infty&\ \mbox{otherwise},
\end{array}
\right.$$
that is a compound Poisson distribution with Gamma distributed jumps;
$$\mbox{if}\ \gamma\in(0,1),\quad\kappa_{(\gamma,\lambda,\theta,0)}(y):=\log\mathbb{E}[e^{yS_{(\gamma,\lambda,\theta,0)}(1)}]
=\left\{\begin{array}{ll}
\lambda(\theta^\gamma-(\theta-y)^\gamma)&\ \mbox{if}\ y\leq\theta\\
\infty&\ \mbox{otherwise},
\end{array}
\right.$$
that is the tempered positive Linnik distribution (actually we have 
the tempered case if $\theta>0$). In view of the application of the G\"artner Ellis
Theorem, we can get the full LDP if the function
$\kappa_{(\gamma,\lambda,\theta,0)}$ is steep; then, in both cases $\gamma\in(0,1]$
and $\gamma\in(-\infty,0)$, we need to check the condition
$\lim_{y\to\theta^-}\kappa_{(\gamma,\lambda,\theta,0)}^\prime(y)=\infty$
and this can be easily done (the details are omitted).

\paragraph{Case $\delta>0$.}
We construct this case starting from the previous one and by considering a Gamma
subordination, i.e. a randomization of the parameter $\lambda$ with a Gamma 
distributed random variable $G_{\delta,\lambda}$ such that
$$\mathbb{E}[e^{yG_{\delta,\lambda}}]=(1-\lambda\delta y)^{-1/\delta}=\left(\frac{(\lambda\delta)^{-1}}{(\lambda\delta)^{-1}-y}\right)^{1/\delta}$$
if $y<(\lambda\delta)^{-1}$, and equal to infinity otherwise. Then, by taking into
account the moment generating function of the random variable $G_{\delta,\lambda}$,
we have
$$\mathbb{E}[e^{yS_{(\gamma,\lambda,\theta,\delta)}(1)}]:=\mathbb{E}[e^{\kappa_{(\gamma,1,\theta,0)}(y)G_{\delta,\lambda}}]
=(1-\delta\kappa_{(\gamma,\lambda,\theta,0)}(y))^{-1/\delta}
=(1-\lambda\delta\hspace{1.0pt}\sgn{\gamma}(\theta^\gamma-(\theta-y)^\gamma))^{-1/\delta}$$
if $y\leq\theta$ and $\sgn{\gamma}(\theta^\gamma-(\theta-y)^\gamma)<(\lambda\delta)^{-1}$,
and equal to infinity otherwise. Thus, for the same values of $y$, the function 
$\kappa_{(\gamma,\lambda,\theta,\delta)}$ is defined by
$$\kappa_{(\gamma,\lambda,\theta,\delta)}:=-\frac{1}{\delta}\log(1-\lambda\delta\hspace{1.0pt}\sgn{\gamma}(\theta^\gamma-(\theta-y)^\gamma)).$$
Moreover, let $y_0$ be the abscissa of convergence of the function 
$\kappa_{(\gamma,\lambda,\theta,\delta)}$, and therefore we have
$\kappa_{(\gamma,\lambda,\theta,\delta)}(y)<\infty$ for $y<y_0$ and 
$\kappa_{(\gamma,\lambda,\theta,\delta)}(y)<\infty$ for $y>y_0$. We remark that
$y_0\in[0,\theta]$. Then we can easily check the steepness of 
$\kappa_{(\gamma,\lambda,\theta,\delta)}$, i.e.
$$\kappa_{(\gamma,\lambda,\theta,\delta)}^\prime(y)
=\frac{\kappa_{(\gamma,\lambda,\theta,0)}^\prime(y)}{1-\delta\kappa_{(\gamma,\lambda,\theta,0)}(y)}
\to\infty\ \mbox{as}\ y\to y_0^-.$$

Finally, as pointed out in \cite{BCCP} (see just after equation (11)), we note that
$$\lim_{\delta\to 0^+}\mathbb{E}[e^{yS_{(\gamma,\lambda,\theta,\delta)}(1)}]
=\lambda\hspace{1.0pt}\sgn{\gamma}(\theta^\gamma-(\theta-y)^\gamma)$$
for all $y\leq\theta$, and equal to infinity otherwise. Therefore we recover the case
$\delta=0$ by taking the limit as $\delta\to 0^+$; in fact the random variable 
$G_{\delta,\lambda}$ converges weakly to the constant $\lambda$ as $\delta\to 0^+$.

\paragraph{Some further comments on both cases $\delta=0$ and $\delta>0$, for $\theta>0$.}
Let $\theta>0$ be fixed. Then, for all $\delta\geq 0$,
$\frac{S_{(\gamma,\lambda,\theta,\delta)}(t)}{t}\to\kappa_{(\gamma,\lambda,\theta,\delta)}^\prime(0)=\lambda\gamma\theta^{\gamma-1}$
as $t\to\infty$ because the rate function $\kappa_{(\gamma,\lambda,\theta,\delta)}^*(x)$ uniquely vanishes
at $x=\kappa_{(\gamma,\lambda,\theta,\delta)}^\prime(0)$. Thus, since the limit value 
$\kappa_{(\gamma,\lambda,\theta,\delta)}^\prime(0)$ does not depend on $\delta$, it is interesting to 
see how the rate function $\kappa_{(\gamma,\lambda,\theta,\delta)}^*(x)$ varies with $\delta$ around the
limit value. In fact, if there exists $\rho>0$ small enough such that
$$\kappa_{(\gamma,\lambda,\theta,\delta_1)}^*(x)>\kappa_{(\gamma,\lambda,\theta,\delta_2)}^*(x)\ \mbox{for}\ 0<|x-\lambda\gamma\theta^{\gamma-1}|<\rho,$$
we can say that $\frac{S_{(\gamma,\lambda,\theta,\delta_1)}(t)}{t}$ converges faster than
$\frac{S_{(\gamma,\lambda,\theta,\delta_2)}(t)}{t}$.

In view of what follows, we remark that
$$\kappa_{(\gamma,\lambda,\theta,\delta)}^{\prime\prime}(0)
=\frac{\kappa_{(\gamma,\lambda,\theta,0)}^{\prime\prime}(0)(1-\delta\kappa_{(\gamma,\lambda,\theta,0)}(0))
+\delta(\kappa_{(\gamma,\lambda,\theta,0)}^\prime(0))^2}{(1-\delta\kappa_{(\gamma,\lambda,\theta,0)}(0))^2}
=\kappa_{(\gamma,\lambda,\theta,0)}^{\prime\prime}(0)+\delta(\kappa_{(\gamma,\lambda,\theta,0)}^\prime(0))^2$$
for all $\delta>0$; actually we can also take $\delta=0$ as a trivial equality. Thus, for
$0\leq\delta_1<\delta_2$, we have
$$\kappa_{(\gamma,\lambda,\theta,\delta_1)}^{\prime\prime}(0)<\kappa_{(\gamma,\lambda,\theta,\delta_2)}^{\prime\prime}(0)$$ 
and the above local inequality between $\kappa_{(\gamma,\lambda,\theta,\delta_1)}^*(x)$ and
$\kappa_{(\gamma,\lambda,\theta,\delta_2)}^*(x)$ holds by some properties of the Fenchel-Legendre transform.
We also remark that our conclusion has some relationship with the Gamma subordination explained above for 
the case $\delta>0$; in fact we have $\mathrm{Var}[G_{\delta,\lambda}]=\lambda^2\delta$, and therefore 
$\delta_1<\delta_2$ yields $\mathrm{Var}[G_{\delta_1,\lambda}]<\mathrm{Var}[G_{\delta_2,\lambda}]$.

Finally we note that, if we take $\delta_2>0$, we get
$$\kappa_{(\gamma,\lambda,\theta,\delta_2)}(y)=-\frac{1}{\delta_2}\log(1-\delta_2\kappa_{(\gamma,\lambda,\theta,0)}(y))
\geq\kappa_{(\gamma,\lambda,\theta,0)}(y)$$
for all $y$ such that $\kappa_{(\gamma,\lambda,\theta,0)}(y)<\frac{1}{\delta_2}$. Then, if $\delta_1=0$, the
above local inequality between $\kappa_{(\gamma,\lambda,\theta,\delta_1)}^*(x)$ and 
$\kappa_{(\gamma,\lambda,\theta,\delta_2)}^*(x)$ holds for all $x\neq\lambda\gamma\theta^{\gamma-1}$.

\subsection{Some remarks}\label{sub:remarks-in-preliminaries}
Here we present some other minor remarks on the family of subordinators studied in this paper.

\begin{remark}[Composition of independent processes]\label{rem:composition-of-independent-subordinators}
It is well-known (and it is easy to check) that, if we consider $h$ independent 
subordinators 
$\{\{S_{(\gamma_i,\lambda_i,\theta_i,\delta_i)}(t):t\geq 0\}:i\in\{1,\ldots,h\}\}$,
the process $\{S(t):t\geq 0\}$ defined by
$$S(t):=S_{(\gamma_1,\lambda_1,\theta_1,\delta_1)}\circ\cdots\circ
S_{(\gamma_n,\lambda_n,\theta_n,\delta_h)}(t)$$
is a subordinator and, moreover, for all $t\geq 0$ we have
$$\mathbb{E}[e^{yS(t)}]=e^{t\kappa_S(y)},\ \mbox{where}\ 
\kappa_S(y):=\kappa_{(\gamma_h,\lambda_h,\theta_h,\delta_h)}\circ\cdots\circ\kappa_{(\gamma_1,\lambda_1,\theta_1,\delta_1)}(y).$$
A natural question is whether, in some cases, the composition of independent processes 
in this family still belongs to this family. One can check that this is possible 
only in a very particular case, i.e.
$$(\gamma_i,\lambda_i,\theta_i,\delta_i)=(\gamma_i,1,0,0)\ \mbox{with}\ \gamma_i\in(0,1),\ \mbox{for all}\ i\in\{1,\ldots,h\},$$
and we have
$$\kappa_{(\gamma_h,1,0,0)}\circ\cdots\circ\kappa_{(\gamma_1,1,0,0)}(y)=\kappa_{(\gamma_1\cdots \gamma_h,1,0,0)}.$$
\end{remark}

\begin{remark}[Generalization of the mixtures in \cite{GuptaKumarLeonenko}]\label{rem:generalization-of-mixtures-in-the-literature}
We consider $h$ independent subordinators 
$\{\{S_{(\gamma_i,\lambda_i,\theta_i,\delta_i)}(t):t\geq 0\}:i\in\{1,\ldots,h\}\}$
and, for some $c_1,\ldots,c_h>0$, let $\{S(t):t\geq 0\}$ be the 
process defined by
$$S(t):=\sum_{i=1}^hS_{(\gamma_i,\lambda_i,\theta_i,\delta_i)}(c_it).$$
Note that this kind of processes is a generalization of the mixtures 
studied in \cite{GuptaKumarLeonenko}; actually in that reference the 
authors require some unnecessary restrictions on the parameters (in 
particular the condition $c_1+\cdots+c_h=1$ that explains the term 
\emph{mixture} used in \cite{GuptaKumarLeonenko}).
For all $t\geq 0$ we have
$$\mathbb{E}[e^{yS(t)}]=e^{t\kappa_S(y)},\ \mbox{where}\ \kappa_S(y):=\sum_{i=1}^hc_i\kappa_{(\gamma_i,\lambda_i,\theta_i,\delta_i)}(y).$$
A natural question is whether, in some cases, the generalized mixture of 
processes in this family (according to the terminology here) still belongs to 
this family. One can check that this is possible in a very particular case,
i.e.
$$(\gamma_i,\lambda_i,\theta_i,\delta_i)=(\gamma,\lambda_i,0,0)\ \mbox{for all}\ i\in\{1,\ldots,h\},\ \mbox{for some}\ \gamma\in(0,1),$$
and we have
$$\kappa_S(y)=\sum_{i=1}^hc_i\kappa_{(\gamma,\lambda_i,0,0)}(y)=\sum_{i=1}^hc_i\lambda_i\kappa_{(\gamma,1,0,0)}(y).$$
\end{remark}

\section{Non-central moderate deviations (for $\delta=0$)}\label{sec:MD}
The term moderate deviations is used in the literature for a suitable class 
of LDPs governed by the same rate function; moreover, in some sense, moderate 
deviations fill the gap between a convergence to a constant and a weak 
convergence to a Gaussian distribution (see e.g. Theorem 3.7.1 in 
\cite{DemboZeitouni} which concerns the case of empirical means of i.i.d.
random vectors, and we can refer to the Law of Large Numbers and to the
Central Limit Theorem).

In this section we study a non-central moderate deviation regime for
$\{S_{(\gamma,\lambda,\theta,0)}(t):t\geq 0\}$ with respect to 
$\theta$; as we said above, we use the term non-central because we deal with 
a non-Gaussian weak limit. Here we deal with finite families of increments
of the subordinator; however, as we shall explain in Remark 
\ref{rem:increments-vs-marginals}, it is also possible to present analogue
results for the finite dimensional distributions of the subordinator.

We start with by considering a family of identically distributed random 
variables (and therefore they are trivially weak convergent).

\begin{proposition}\label{prop:identically-distributed-rvs}
Let $m\geq 1$ and $0=t_0<t_1<t_2<\cdots<t_m$ be arbitrarily fixed. Then, for
all $\theta>0$, the random vector $(\theta S_{(\gamma,\lambda,\theta,0)}(t_i/\theta^\gamma)
-\theta S_{(\gamma,\lambda,\theta,0)}(t_{i-1}/\theta^\gamma))_{i=1,\ldots,m}$ is distributed
as $(S_{(\gamma,\lambda,1,0)}(t_i)-S_{(\gamma,\lambda,1,0)}(t_{i-1}))_{i=1,\ldots,m}$. Thus
$$\{(\theta S_{(\gamma,\lambda,\theta,0)}(t_i/\theta^\gamma)
-\theta S_{(\gamma,\lambda,\theta,0)}(t_{i-1}/\theta^\gamma))_{i=1,\ldots,m}:\theta>0\}$$
is a family of identically distributed random vectors.
\end{proposition}
\begin{proof}
By taking into account the independence and the distribution of the increments, for all 
$\theta>0$ we have
\begin{multline*}
\log\mathbb{E}\left[\exp\left(\sum_{i=1}^my_i(\theta S_{(\gamma,\lambda,\theta,0)}(t_i/\theta^\gamma)
-\theta S_{(\gamma,\lambda,\theta,0)}(t_{i-1}/\theta^\gamma))\right)\right]\\
=\sum_{i=1}^m\log\mathbb{E}\left[e^{\theta y_i(S_{(\gamma,\lambda,\theta,0)}(t_i/\theta^\gamma)
-S_{(\gamma,\lambda,\theta,0)}(t_{i-1}/\theta^\gamma))}\right]
=\sum_{i=1}^m\frac{t_i-t_{i-1}}{\theta^\gamma}\kappa_{(\gamma,\lambda,\theta,0)}(\theta y_i)\\
=\left\{\begin{array}{ll}
\sum_{i=1}^m\frac{t_i-t_{i-1}}{\theta^\gamma}\lambda\hspace{1.0pt}\sgn{\gamma}(\theta^\gamma-(\theta-\theta y_i)^\gamma)
&\ \mbox{if}\ \theta y_1,\ldots,\theta y_m\leq\theta\\
\infty&\ \mbox{otherwise}
\end{array}\right.\\
=\left\{\begin{array}{ll}
\sum_{i=1}^m(t_i-t_{i-1})\lambda\hspace{1.0pt}\sgn{\gamma}(1-(1-y_i)^\gamma)&\ \mbox{if}\ y_1,\ldots,y_m\leq 1\\
\infty&\ \mbox{otherwise}
\end{array}\right.
=\sum_{i=1}^m(t_i-t_{i-1})\kappa_{(\gamma,\lambda,1,0)}(y_i).
\end{multline*}
This completes the proof.
\end{proof}

The result stated in Proposition \ref{prop:identically-distributed-rvs} allows to consider different kind 
of weak convergence. Here we mainly consider the case $\theta\to\infty$; the case $\theta\to 0$ will be 
briefly discussed in Remark \ref{rem:case-theta-to-zero}.

\begin{proposition}\label{prop:LD-theta-to-infinity}
Let $m\geq 1$ and $0=t_0<t_1<t_2<\cdots<t_m$ be arbitrarily fixed. Moreover	
let $g(\gamma),h(\gamma)\in\mathbb{R}$ be such that $\gamma-h(\gamma)=1-g(\gamma)>0$. Then the family 
of random vectors
$$\{(\theta^{g(\gamma)}S_{(\gamma,\lambda,\theta,0)}(t_i/\theta^{h(\gamma)})
-\theta^{g(\gamma)}S_{(\gamma,\lambda,\theta,0)}(t_{i-1}/\theta^{h(\gamma)}))_{i=1,\ldots,m}:\theta>0\}$$
satisfies the LDP with speed $\theta^{\gamma-h(\gamma)}$, or equivalently $\theta^{1-g(\gamma)}$, and good
rate function $I_{t_1,\ldots,t_m}$ defined by
$$I_{t_1,\ldots,t_m}(x_1,\ldots,x_m)=\sum_{i=1}^m(t_i-t_{i-1})\kappa_{(\gamma,\lambda,1,0)}^*\left(\frac{x_i}{t_i-t_{i-1}}\right).$$
\end{proposition}
\begin{proof}
We want to apply the G\"artner Ellis Theorem. Firstly, we have
\begin{multline*}
\frac{1}{\theta^{\gamma-h(\gamma)}}\log\mathbb{E}\left[\exp\left(\theta^{\gamma-h(\gamma)}
\sum_{i=1}^m y_i(\theta^{g(\gamma)}S_{(\gamma,\lambda,\theta,0)}(t_i/\theta^{h(\gamma)})
-\theta^{g(\gamma)}S_{(\gamma,\lambda,\theta,0)}(t_{i-1}/\theta^{h(\gamma)}))\right)\right]\\
=\frac{1}{\theta^{\gamma-h(\gamma)}}\sum_{i=1}^m\log\mathbb{E}\left[e^{\theta^{\gamma-h(\gamma)+g(\gamma)}
y_i(S_{(\gamma,\lambda,\theta,0)}(t_i/\theta^{h(\gamma)})
-S_{(\gamma,\lambda,\theta,0)}(t_{i-1}/\theta^{h(\gamma)}))}\right]\\
=\sum_{i=1}^m\frac{t_i-t_{i-1}}{\theta^{\gamma-h(\gamma)+h(\gamma)}}\kappa_{(\gamma,\lambda,\theta,0)}(\theta^{\gamma-h(\gamma)+g(\gamma)}y_i)
=\sum_{i=1}^m\frac{t_i-t_{i-1}}{\theta^\gamma}\kappa_{(\gamma,\lambda,\theta,0)}(\theta y_i).
\end{multline*}
Moreover, by taking into account some computations in the proof of Proposition
\ref{prop:identically-distributed-rvs}, the final expression does not depend
on $\theta$, and we have
\begin{equation}\label{eq:theta-invariance}
\sum_{i=1}^m\frac{t_i-t_{i-1}}{\theta^\gamma}\kappa_{(\gamma,\lambda,\theta,0)}(\theta y_i)
=\sum_{i=1}^m(t_i-t_{i-1})\kappa_{(\gamma,\lambda,1,0)}(y_i)\ \mbox{for all}\ \theta>0.
\end{equation}
Then, for all $(y_1,\ldots,y_m)\in\mathbb{R}^m$, we have
\begin{multline*}
\lim_{\theta\to\infty}\frac{1}{\theta^{\gamma-h(\gamma)}}\log\mathbb{E}\left[\exp\left(\theta^{\gamma-h(\gamma)}
\sum_{i=1}^m y_i(\theta^{g(\gamma)}S_{(\gamma,\lambda,\theta,0)}(t_i/\theta^{h(\gamma)})
-\theta^{g(\gamma)}S_{(\gamma,\lambda,\theta,0)}(t_{i-1}/\theta^{h(\gamma)}))\right)\right]\\
=\sum_{i=1}^m(t_i-t_{i-1})\kappa_{(\gamma,\lambda,1,0)}(y_i).
\end{multline*}
So we can apply the G\"artner Ellis Theorem and the desired LDP holds with good
rate function $I_{t_1,\ldots,t_m}$ defined by
$$I_{t_1,\ldots,t_m}(x_1,\ldots,x_m)=\sup_{(y_1,\ldots,y_m)\in\mathbb{R}^m}
\left\{\sum_{i=1}^my_ix_i-\sum_{i=1}^m(t_i-t_{i-1})\kappa_{(\gamma,\lambda,1,0)}(y_i)\right\}.$$
Finally, one can check that the rate function $I_{t_1,\ldots,t_m}$ defined here 
coincides with the one in the statement of the proposition; this can be done 
with some standard computations (for instance one can follow the lines of the proof 
of Lemma 5.1.8 in \cite{DemboZeitouni}).
\end{proof}

For completeness we discuss the convergence of the random variables in Proposition
\ref{prop:LD-theta-to-infinity} to a constant vector. Firstly, for $\theta$ 
large enough which depends on $y_1,\ldots,y_m$ (otherwise the moment generating 
function below is equal to infinity), we have
\begin{multline*}
\log\mathbb{E}\left[\exp\left(\sum_{i=1}^m y_i(\theta^{g(\gamma)}
S_{(\gamma,\lambda,\theta,0)}(t_i/\theta^{h(\gamma)})
-\theta^{g(\gamma)}S_{(\gamma,\lambda,\theta,0)}(t_{i-1}/\theta^{h(\gamma)}))\right)\right]\\
=\sum_{i=1}^m\frac{t_i-t_{i-1}}{\theta^{h(\gamma)}}\kappa_{(\gamma,\lambda,\theta,0)}(\theta^{g(\gamma)}y_i)
=\sum_{i=1}^m\frac{t_i-t_{i-1}}{\theta^{h(\gamma)}}\lambda\hspace{1.0pt}\sgn{\gamma}(\theta^\gamma-(\theta-\theta^{g(\gamma)}y_i)^\gamma)\\
=\sum_{i=1}^m(t_i-t_{i-1})\lambda\hspace{1.0pt}\sgn{\gamma}\theta^{\gamma-h(\gamma)}\left(1-\left(1-\frac{y_i}{\theta^{1-g(\gamma)}}\right)^\gamma\right).
\end{multline*}
Then
\begin{multline*}
\lim_{\theta\to\infty}\log\mathbb{E}\left[\exp\left(\sum_{i=1}^m y_i(\theta^{g(\gamma)}
S_{(\gamma,\lambda,\theta,0)}(t_i/\theta^{h(\gamma)})
-\theta^{g(\gamma)}S_{(\gamma,\lambda,\theta,0)}(t_{i-1}/\theta^{h(\gamma)}))\right)\right]\\
=\sum_{i=1}^my_i(t_i-t_{i-1})\lambda\hspace{1.0pt}\sgn{\gamma}\gamma,
\end{multline*}
so that the random variables in Proposition \ref{prop:LD-theta-to-infinity} converge (as 
$\theta\to\infty$) to the vector $(x_1(\gamma),\ldots,x_m(\gamma))$ defined by
$$x_i(\gamma)=(t_i-t_{i-1})\lambda\hspace{1.0pt}\sgn{\gamma}\gamma\ (\mbox{for all}\ i\in\{1,\ldots,m\}).$$
Moreover, as one can expect, $I_{t_1,\ldots,t_m}(x_1,\ldots,x_m)=0$ if and only if 
\begin{multline*}
(x_1,\ldots,x_m)=\left(\left.\frac{\partial}{\partial y_i}\sum_{i=1}^m(t_i-t_{i-1})
\kappa_{(\gamma,\lambda,1,0)}(y_i)\right|_{(y_1,\ldots,y_m)=(0,\ldots,0)}\right)_{i=1,\ldots,m}\\
=((t_i-t_{i-1})\kappa_{(\gamma,\lambda,1,0)}^\prime(0))_{i=1,\ldots,m}=
(x_1(\gamma),\ldots,x_m(\gamma)).
\end{multline*}

Now we are ready to present the non-central moderate deviation 
result (as $\theta\to\infty$); see also Remark 
\ref{rem:MD-typical-features}.

\begin{proposition}\label{prop:MD-theta-to-infinity}
Let $m\geq 1$ and $0=t_0<t_1<t_2<\cdots<t_m$ be arbitrarily fixed. Moreover	
let $g(\gamma),h(\gamma)$ be such that $\gamma-h(\gamma)=1-g(\gamma)>0$ (as in
Proposition \ref{prop:LD-theta-to-infinity}). Then, for all families of positive 
numbers $\{a_\theta:n\geq 1\}$ such that
\begin{equation}\label{eq:MD-conditions-theta-to-infinity}
a_\theta\to 0\ \mbox{and}\ \theta^{\gamma-h(\gamma)}a_\theta=\theta^{1-g(\gamma)}a_\theta\to\infty\ (\mbox{as}\ \theta\to\infty),
\end{equation}
the family of random vectors
$$\{(a_\theta\theta S_{(\gamma,\lambda,\theta,0)}(t_i/(a_\theta\theta^\gamma))-a_\theta\theta S_{(\gamma,\lambda,\theta,0)}(t_{i-1}/(a_\theta\theta^\gamma)))_{i=1,\ldots,m}:\theta>0\}$$
satisfies the LDP with speed $1/a_\theta$ and good rate function $I_{t_1,\ldots,t_m}$
presented in Proposition \ref{prop:LD-theta-to-infinity}.
\end{proposition}
\begin{proof}
We want to apply the G\"artner Ellis Theorem. For all $\theta>0$, by taking into account
equation \eqref{eq:theta-invariance} for the last equality below, we get
\begin{multline*}
\frac{1}{1/a_\theta}\log\mathbb{E}\left[\exp\left(\frac{1}{a_\theta}\sum_{i=1}^m y_i
(a_\theta\theta S_{(\gamma,\lambda,\theta,0)}(t_i/(a_\theta\theta^\gamma))-a_\theta\theta S_{(\gamma,\lambda,\theta,0)}(t_{i-1}/(a_\theta\theta^\gamma)))\right)\right]\\
=a_\theta\sum_{i=1}^m\log\mathbb{E}\left[e^{\theta y_i\{S_{(\gamma,\lambda,\theta,0)}(t_i/(a_\theta\theta^\gamma))
-S_{(\gamma,\lambda,\theta,0)}(t_{i-1}/(a_\theta\theta^\gamma))\}}\right]\\
=a_\theta\sum_{i=1}^m\frac{t_i-t_{i-1}}{a_\theta\theta^\gamma}\kappa_{(\gamma,\lambda,\theta,0)}(\theta y_i)
=\sum_{i=1}^m\frac{t_i-t_{i-1}}{\theta^\gamma}\kappa_{(\gamma,\lambda,\theta,0)}(\theta y_i)
=\sum_{i=1}^m(t_i-t_{i-1})\kappa_{(\gamma,\lambda,1,0)}(y_i);
\end{multline*}
so, for all $(y_1,\ldots,y_m)\in\mathbb{R}^m$, we have
\begin{multline*}
\lim_{\theta\to\infty}\frac{1}{1/a_\theta}\log\mathbb{E}\left[\exp\left(\frac{1}{a_\theta}\sum_{i=1}^m y_i
(a_\theta\theta S_{(\gamma,\lambda,\theta,0)}(t_i/(a_\theta\theta^\gamma))-a_\theta\theta S_{(\gamma,\lambda,\theta,0)}(t_{i-1}/(a_\theta\theta^\gamma)))\right)\right]\\
=\sum_{i=1}^m(t_i-t_{i-1})\kappa_{(\gamma,\lambda,1,0)}(y_i).
\end{multline*}
We conclude the proof by considering the same application of the
G\"artner Ellis Theorem presented in the proof of Proposition
\ref{prop:LD-theta-to-infinity}.
\end{proof}

We conclude with some remarks.

\begin{remark}\label{rem:MD-typical-features}
The class of LDPs in Proposition \ref{prop:MD-theta-to-infinity} fill the
gap between the following asymptotic regimes:
\begin{itemize}
\item the convergence of the random variables in Proposition
\ref{prop:LD-theta-to-infinity} to $(x_1(\gamma),\ldots,x_m(\gamma))$;
\item the weak convergence of the random variables in Proposition
\ref{prop:identically-distributed-rvs} that trivially converge to
their common law, and therefore the law of the random vector
$(S_{(\gamma,\lambda,1,0)}(t_i)-S_{(\gamma,\lambda,1,0)}(t_{i-1}))_{i=1,\ldots,m}$.
\end{itemize}
In some sense these two asymptotic regimes can be recovered by considering
two extremal choices for $a_\theta$ in Proposition 
\ref{prop:MD-theta-to-infinity}, i.e. 
$a_\theta=\frac{1}{\theta^{\gamma-h(\gamma)}}=\frac{1}{\theta^{1-g(\gamma)}}$
and $a_\theta=1$, respectively. Note that, in both cases, one condition in 
\eqref{eq:MD-conditions-theta-to-infinity} holds and the other one fails.
\end{remark}

\begin{remark}\label{rem:MD-comments-on-rf}
The rate function $I_{t_1,\ldots,t_m}$ in Propositions \ref{prop:LD-theta-to-infinity}
and \ref{prop:MD-theta-to-infinity} has some connections with the two asymptotic 
regimes as $\theta\to\infty$ presented in Remark \ref{rem:MD-typical-features}.
\begin{itemize}
\item The rate function $I_{t_1,\ldots,t_m}$ uniquely vanishes at
$(x_1(\gamma),\ldots,x_m(\gamma))$ and, as already remarked, this vector is
the limit of the random variables in Proposition \ref{prop:LD-theta-to-infinity}
as $\theta\to\infty$.
\item The Hessian matrix $\left(\left.\frac{\partial^2}{\partial x_i\partial x_j}I_{t_1,\ldots,t_m}
(x_1,\ldots,x_m)\right|_{(x_1,\ldots,x_m)=(x_1(\gamma),\ldots,x_m(\gamma))}\right)_{i,j=1,\ldots,m}$ has some
connections with the law of the random vector
$(S_{(\gamma,\lambda,1,0)}(t_i)-S_{(\gamma,\lambda,1,0)}(t_{i-1}))_{i=1,\ldots,m}$ that
appears in Proposition \ref{prop:identically-distributed-rvs}. More precisely it is
a diagonal matrix by the independence of the increments and, as far as the diagonal entries
are concerned, we have
\begin{multline*}
\left.\frac{\partial^2}{\partial x_i^2}I_{t_1,\ldots,t_m}
(x_1,\ldots,x_m)\right|_{(x_1,\ldots,x_m)=(x_1(\gamma),\ldots,x_m(\gamma))}\\
=\frac{1}{(t_i-t_{i-1})\mathrm{Var}[S_{(\gamma,\lambda,1,0)}(1)]}
=\frac{1}{\mathrm{Var}[S_{(\gamma,\lambda,1,0)}(t_i)-S_{(\gamma,\lambda,1,0)}(t_{i-1})]}
\ (\mbox{for all}\ i\in\{1,\ldots,m\}).
\end{multline*}
\end{itemize}
\end{remark}

\begin{remark}\label{rem:increments-vs-marginals}
The results presented in this section concern the increments of the process.
However we can derive analogue results for the finite dimensional 
distributions of the subordinator. The idea is to combine the above propositions
and a suitable transformation of the involved random vectors with the continuous
function
$$(x_1,\ldots,x_m)\mapsto f(x_1,\ldots,x_m):=\left(x_1,x_1+x_2,\ldots,\sum_{i=1}^m x_i\right).$$
In particular, as far as Propositions \ref{prop:LD-theta-to-infinity} and
\ref{prop:MD-theta-to-infinity} are concerned, we can apply the contraction
principle recalled in Section \ref{sec:preliminaries-LD}. Then we have the 
following statements.
\begin{itemize}
\item For all $\theta>0$, the random vector $\theta S_{(\gamma,\lambda,\theta,0)}(t_i/\theta^\gamma))_{i=1,\ldots,m}$
is distributed as $S_{(\gamma,\lambda,1,0)}(t_i))_{i=1,\ldots,m}$; therefore
$\{(\theta S_{(\gamma,\lambda,\theta,0)}(t_i/\theta^\gamma))_{i=1,\ldots,m}:\theta>0\}$
is a family of identically distributed random vectors.
\item The family of random vectors
$\{(\theta^{g(\gamma)}S_{(\gamma,\lambda,\theta,0)}(t_i/\theta^{h(\gamma)}))_{i=1,\ldots,m}:\theta>0\}$
satisfies the LDP with speed $\theta^{\gamma-h(\gamma)}$, or equivalently 
$\theta^{1-g(\gamma)}$, and good rate function $J_{t_1,\ldots,t_m}$ 
defined by
\begin{multline*}
J_{t_1,\ldots,t_m}(z_1,\ldots,z_m)=\inf\{I_{t_1,\ldots,t_m}(x_1,\ldots,x_m):
f(x_1,\ldots,x_m)=(z_1,\ldots,z_m)\}\\
=I_{t_1,\ldots,t_m}(z_1,z_2-z_1,\ldots,z_m-z_{m-1})=
\sum_{i=1}^m(t_i-t_{i-1})\kappa_{(\gamma,\lambda,1,0)}^*\left(\frac{z_i-z_{i-1}}{t_i-t_{i-1}}\right),
\end{multline*}
where $z_0=0$ in the last equality.
\item If condition \eqref{eq:MD-conditions-theta-to-infinity} holds, then the
random vectors
$\{(a_\theta\theta S_{(\gamma,\lambda,\theta,0)}(t_i/(a_\theta\theta^\gamma)))_{i=1,\ldots,m}:\theta>0\}$
satisfy the LDP with speed $1/a_\theta$ and good rate function $J_{t_1,\ldots,t_m}$
defined in the item above.
\end{itemize}
\end{remark}

\begin{remark}\label{rem:case-theta-to-zero}
All the results above and Remark \ref{rem:increments-vs-marginals} concern the case 
$\theta\to\infty$. In order to obtain the analogue versions for the case $\theta\to 0$
some changes are needed. Proposition \ref{prop:identically-distributed-rvs} is 
essentially without changes because the involved random variables are identically 
distributed. The condition $\gamma-h(\gamma)=1-g(\gamma)>0$ in Propositions 
\ref{prop:LD-theta-to-infinity} and \ref{prop:MD-theta-to-infinity} 
(and in Remark \ref{rem:increments-vs-marginals}) has to be replaced with 
$\gamma-h(\gamma)=1-g(\gamma)<0$. The speed function is again
$\theta^{\gamma-h(\gamma)}=\theta^{1-g(\gamma)}$, which tends to infinity as 
$\theta\to 0$ because $\gamma-h(\gamma)=1-g(\gamma)<0$. Condition 
\eqref{eq:MD-conditions-theta-to-infinity} in Proposition 
\ref{prop:MD-theta-to-infinity} has to be replaced with
\begin{equation}\label{eq:MD-conditions-theta-to-0}
a_\theta\to 0\ \mbox{and}\ \theta^{\gamma-h(\gamma)}a_\theta=\theta^{1-g(\gamma)}a_\theta\to\infty\ (\mbox{as}\ \theta\to 0);
\end{equation}
note that, in both conditions \eqref{eq:MD-conditions-theta-to-infinity} and
\eqref{eq:MD-conditions-theta-to-0}, one requires that $a_\theta$ tends to zero
slowly. We can also present a version of Remark \ref{rem:MD-typical-features}. 
Firstly Proposition \ref{prop:MD-theta-to-infinity} for the case $\theta\to 0$ 
provides a class of LDPs which fill the gap between a convergence to 
$(x_1(\gamma),\ldots,x_m(\gamma))$, and a trivial weak convergence as $\theta\to 0$
for a family of identically distributed random variables (they can be derived by
Propositions \ref{prop:LD-theta-to-infinity} for the case $\theta\to 0$, and
Proposition \ref{prop:identically-distributed-rvs}). Moreover the convergence to a 
constant and the trivial weak convergence correspond to the cases 
$a_\theta=\theta^{-(\gamma-h(\gamma))}=\theta^{-(1-g(\gamma))}$ and $a_\theta=1$; so,
in both cases, one condition in \eqref{eq:MD-conditions-theta-to-0} holds and the other
one fails. Finally, since the rate functions in Propositions \ref{prop:LD-theta-to-infinity}
and \ref{prop:MD-theta-to-infinity} (and in Remark \ref{rem:increments-vs-marginals}) for 
$\theta\to 0$ coincide with the ones for the case $\theta\to\infty$, we can repeat the 
comments in Remark \ref{rem:MD-comments-on-rf} without any changes.
\end{remark}

\begin{remark}\label{rem:self-similarity}
It is possible to consider a more general version of Proposition \ref{prop:identically-distributed-rvs} 
by replacing the process $\{S_{(\gamma,\lambda,\theta,0)}(t):t\geq 0\}$ with a more general 
self-similar process $\{S(t):t\geq 0\}$ with having independent and stationary increments.
More precisely let $\{S(t):t\geq 0\}$ be a self-similar process with index $H>0$ and, again, let
$m\geq 1$ and $0=t_0<t_1<t_2<\cdots<t_m$ be arbitrarily fixed. Then, for all $\theta>0$, the 
random vector $(\theta S(t_i/\theta^{1/H})-\theta S(t_{i-1}/\theta^{1/H}))_{i=1,\ldots,m}$ 
is distributed as $(S(t_i)-S(t_{i-1}))_{i=1,\ldots,m}$. Furthermore a weaker version of this result,
with $m=1$ only, could be considered if we do not require any hypotheses on the increments.
\end{remark}

\section{Large deviations for inverse processes}\label{sec:LD-inverse-processes}
In this section we consider the inverse of 
$\{S_{(\gamma,\lambda,\theta,\delta)}(t):t\geq 0\}$, i.e. the process
$\{T_{(\gamma,\lambda,\theta,\delta)}(t):t\geq 0\}$ defined by
$$T_{(\gamma,\lambda,\theta,\delta)}(t):=\inf\{u>0:S_{(\gamma,\lambda,\theta,\delta)}(u)>t\}.$$

\begin{remark}\label{rem:lambda-scale-parameter-for-delta=0}
Assume that $\delta=0$. Then $\{S_{(\gamma,\lambda,\theta,0)}(t):t\geq 0\}$ and
$\{S_{(\gamma,1,\theta,0)}(\lambda t):t\geq 0\}$ have the same finite-dimensional 
distributions. Thus $\{T_{(\gamma,\lambda,\theta,0)}(t):t\geq 0\}$ and
$\left\{\frac{T_{(\gamma,1,\theta,0)}(t)}{\lambda}:t\geq 0\right\}$ also have the same 
finite-dimensional distributions; however, in what follows, we only need to consider
their one-dimensional distributions.
\end{remark}

Our aim is to illustrate an application of the results for inverse processes in
\cite{DuffieldWhitt}; actually we always consider the simple case in which the
$u,v,w$ in that reference are the identity function. Moreover, since the speed 
for the LDPs in this section is always $v_t=t$, we omit this detail.

A naive approach is to consider the application of the G\"artner Ellis Theorem to 
$\{T_{(\gamma,\lambda,\theta,\delta)}(t)/t:t>0\}$ as $t\to\infty$; in other words, if there 
exists (for all $y\in\mathbb{R}$)
\begin{equation}\label{eq:GE-limit-naive-approach}
\lim_{t\to\infty}\frac{1}{t}\log\mathbb{E}[e^{yT_{(\gamma,\lambda,\theta,\delta)}(t)}]
=\Lambda_{(\gamma,\lambda,\theta,\delta)}(y)
\end{equation}
as an extended real number, and the function $\Lambda_{(\gamma,\lambda,\theta,\delta)}$
satisfies some conditions, we can say that $\{T_{(\gamma,\lambda,\theta,\delta)}(t)/t:t>0\}$
satisfies the LDP with good rate function 
$\Lambda_{(\gamma,\lambda,\theta,\delta)}^*$ defined by
$$\Lambda_{(\gamma,\lambda,\theta,\delta)}^*(x):=\sup_{y\in\mathbb{R}}\{xy-\Lambda_{(\gamma,\lambda,\theta,\delta)}(y)\}.$$
Unfortunately, in general, the moment generating function 
$\mathbb{E}[e^{yT_{(\gamma,\lambda,\theta,\delta)}(t)}]$ is not available.

The approach based on the application of the results in \cite{DuffieldWhitt} allows to overcome
this problem. In order to do that we consider the LDP of 
$\{S_{(\gamma,\lambda,\theta,\delta)}(t)/t:t>0\}$ as $t\to\infty$, and this can be done by
considering an application of the G\"artner Ellis Theorem because the moment generating function 
$\mathbb{E}[e^{yS_{(\gamma,\lambda,\theta,\delta)}(t)}]$ is available. In fact we have
\begin{equation}\label{eq:GE-limit-subordinator}
\lim_{t\to\infty}\frac{1}{t}\log\mathbb{E}[e^{yS_{(\gamma,\lambda,\theta,\delta)}(t)}]
=\kappa_{(\gamma,\lambda,\theta,\delta)}(y)\ (\mbox{for all}\ y\in\mathbb{R}),
\end{equation}
where the function $\kappa_{(\gamma,\lambda,\theta,\delta)}$ has been introduced in Section 
\ref{sec:preliminaries-BCCP}; moreover, if the function $\kappa_{(\gamma,\lambda,\theta,\delta)}$ 
satisfies some conditions (see the case $\theta>0$ below), the LDP holds with good rate function 
$\kappa_{(\gamma,\lambda,\theta,\delta)}^*$ defined by \eqref{eq:Legendre-transform-kappa}.
Then we can apply the results in \cite{DuffieldWhitt} and we have the following claims.

\begin{claim}\label{claim:DW1}
By Theorem 1(i) in \cite{DuffieldWhitt}, $\{T_{(\gamma,\lambda,\theta,\delta)}(t)/t:t>0\}$
satisfies the LDP with good rate function $\Psi_{(\gamma,\lambda,\theta,\delta)}$ defined by
$$\Psi_{(\gamma,\lambda,\theta,\delta)}(x)=x\kappa_{(\gamma,\lambda,\theta,\delta)}^*(1/x)$$
for $x>0$, $\Psi_{(\gamma,\lambda,\theta,\delta)}(0)=\lim_{x\to 0^+}\Psi_{(\gamma,\lambda,\theta,\delta)}(x)$,
and $\Psi_{(\gamma,\lambda,\theta,\delta)}(x)=\infty$ for $x<0$.
\end{claim}

\begin{claim}\label{claim:DW2}
By Theorem 3(ii) in \cite{DuffieldWhitt} (note that the function $I$ in that reference 
coincides with $\kappa_{(\gamma,\lambda,\theta,\delta)}^*$ in this paper) condition 
\eqref{eq:GE-limit-naive-approach} holds for $y<\kappa_{(\gamma,\lambda,\theta,\delta)}^*(0)$
and we have
\begin{equation}\label{eq:*}
\Lambda_{(\gamma,\lambda,\theta,\delta)}(y)=\sup_{x\in\mathbb{R}}\{xy-\Psi_{(\gamma,\lambda,\theta,\delta)}(x)\};
\end{equation}
moreover we have
$$\kappa_{(\gamma,\lambda,\theta,\delta)}^*(0)=-\lim_{y\to-\infty}\kappa_{(\gamma,\lambda,\theta,\delta)}(y)=
\left\{\begin{array}{ll}
\infty&\ \mbox{if}\ \gamma\in(0,1)\ \mbox{and}\ \delta\geq 0\\
\frac{1}{\delta}\log(1+\lambda\delta\theta^\gamma)&\ \mbox{if}\ \gamma\in(-\infty,0)\ \mbox{and}\ \delta>0\\
\lambda\theta^\gamma&\ \mbox{if}\ \gamma\in(-\infty,0)\ \mbox{and}\ \delta=0.
\end{array}\right.$$
\end{claim}

We also recall that, for all $\delta\geq 0$, we have $\kappa_{(\gamma,\lambda,\theta,\delta)}^*(x)=0$ if and 
only if $x=\kappa_{(\gamma,\lambda,\theta,\delta)}^\prime(0)=\lambda\hspace{1.0pt}\sgn{\gamma}\gamma\theta^{\gamma-1}$; 
thus, by the definition of $\Psi_{(\gamma,\lambda,\theta,\delta)}$, we have
$\Psi_{(\gamma,\lambda,\theta,\delta)}(x)=0$ if and only if
$x=(\kappa_{(\gamma,\lambda,\theta,\delta)}^\prime(0))^{-1}$.

We shall discuss the case $\theta=0$ in Section \ref{sub:theta=0} and the case $\theta>0$ in Section 
\ref{sub:theta>0}. Finally, in Section \ref{sub:GW}, we shall compute the function 
$\Lambda_{(\gamma,\lambda,\theta,\delta)}$ in \eqref{eq:*} when $\delta=0$.

\subsection{Case $\theta=0$}\label{sub:theta=0}
In this case we cannot apply the G\"artner Ellis Theorem to obtain the LDP of $\{S_{(\gamma,\lambda,\theta,\delta)}(t)/t:t>0\}$
as $t\to\infty$; in fact $\kappa_{(\gamma,\lambda,0,\delta)}$ is not finite in a neighborhood of the 
origin. We recall that we only have $\gamma\in(0,1)$ when $\theta=0$. We can obtain the LDP of 
$\{T_{(\gamma,\lambda,0,\delta)}(t)/t:t>0\}$ as $t\to\infty$ only if $\delta=0$. In fact, if we 
consider the Mittag-Leffler function $E_\gamma(x):=\sum_{k=0}^\infty\frac{x^k}{\Gamma(\gamma k+1)}$,
we have
$$\mathbb{E}[e^{yT_{(\gamma,\lambda,0,0)}(t)}]=\mathbb{E}[e^{yT_{(\gamma,1,0,0)}(t)/\lambda}]=E_\gamma\left(\frac{y}{\lambda}t^\gamma\right)$$
by Remark \ref{rem:lambda-scale-parameter-for-delta=0} and by a well-known result in the literature 
for $\lambda=1$ (see e.g. eq. (24) in \cite{MainardiMuraPagnini}, or eq. (16) in \cite{Bingham} for 
the case $y\leq 0$). Then, by taking into account the asymptotic behavior of the Mittag-Leffler 
function as its argument tends to infinity (see e.g. eq. (1.8.27) in \cite{KilbasSrivastavaTrujillo}),
we have
\begin{equation}\label{eq:GE-limit-for-inverse-particular-case}
\lim_{t\to\infty}\frac{1}{t}\log\mathbb{E}[e^{yT_{(\gamma,\lambda,0,0)}(t)}]
=\left\{\begin{array}{ll}
(y/\lambda)^{1/\gamma}&\ \mbox{if}\ y\geq 0\\
0&\ \mbox{if}\ y<0
\end{array}\right.=:\Lambda_{(\gamma,\lambda,0,0)}(y).
\end{equation}
So, in this case, we can consider the naive approach discussed above (see the sentence with equation 
\eqref{eq:GE-limit-naive-approach}). Then the G\"artner Ellis Theorem yields the LDP of 
$\{T_{(\gamma,\lambda,0,0)}(t)/t:t>0\}$ as $t\to\infty$ with good
rate function $\Psi_{(\gamma,\lambda,0,0)}:=\Lambda_{(\gamma,\lambda,0,0)}^*$, i.e.
\begin{equation}\label{eq:Psi-for-theta=delta=0}
\Psi_{(\gamma,\lambda,0,0)}(x):=\sup_{y\in\mathbb{R}}\{xy-\Lambda_{(\gamma,\lambda,0,0)}(y)\}
=\left\{\begin{array}{ll}
\lambda^{1/(1-\gamma)}(\gamma^{\gamma/(1-\gamma)}-\gamma^{1/(1-\gamma)})x^{1/(1-\gamma)}&\ \mbox{if}\ x\geq 0\\
\infty&\ \mbox{if}\ x<0;
\end{array}\right.
\end{equation}
moreover, noting that $\Lambda_{(\gamma,\lambda,0,0)}^\prime(0)=0$, we
have $\Lambda_{(\gamma,\lambda,0,0)}^*(x)=0$ if and only if $x=0$.

\subsection{Case $\theta>0$}\label{sub:theta>0}
In this case we can apply the G\"artner Ellis Theorem by considering the limit in 
\eqref{eq:GE-limit-subordinator}; in fact the function $\kappa_{(\gamma,\lambda,\theta,\delta)}(y)$
is finite in a neighborhood of the origin. However, we cannot provide an explicit expression of 
$\kappa_{(\gamma,\lambda,\theta,\delta)}^*$ and $\Psi_{(\gamma,\lambda,\theta,\delta)}$ when $\delta>0$.
On the contrary, this is feasible when $\delta=0$. In fact, after some easy computations, we get:
$$\kappa_{(\gamma,\lambda,\theta,0)}^*(x)=x\left(\theta-\left(\frac{x}{\lambda\hspace{1.0pt}\sgn{\gamma}\gamma}\right)^{1/(\gamma-1)}\right)
-\lambda\hspace{1.0pt}\sgn{\gamma}\left(\theta^\gamma-
\left(\frac{x}{\lambda\hspace{1.0pt}\sgn{\gamma}\gamma}\right)^{\gamma/(\gamma-1)}\right)\ (\mbox{for all}\ x>0)$$
and $\kappa_{(\gamma,\lambda,\theta,\delta)}^*(0)=\infty$;
\begin{equation}\label{eq:Psi-for-delta=0}
\Psi_{(\gamma,\lambda,\theta,0)}(x)=\theta-(\lambda\hspace{1.0pt}\sgn{\gamma}\gamma x)^{1/(1-\gamma)}
+\lambda\hspace{1.0pt}\sgn{\gamma}x\left((\lambda\hspace{1.0pt}\sgn{\gamma}\gamma x)^{\gamma/(1-\gamma)}
-\theta^\gamma\right)\ (\mbox{for all}\ x\geq 0).
\end{equation}
Moreover, in particular, the right derivative of $\Psi_{(\gamma,\lambda,\theta,0)}$ at $y=0$ is
\begin{equation}\label{eq:right-derivative-of-Psi-at-zero}
\Psi_{(\gamma,\lambda,\theta,0)}^\prime(0)=\left\{\begin{array}{ll}
-\lambda\theta^\gamma&\ \mbox{if}\ \gamma\in(0,1)\\
-\infty&\ \mbox{if}\ \gamma\in(-\infty,0).
\end{array}\right.
\end{equation}
We also remark that, when $\gamma\in(0,1)$, the expression of $\Psi_{(\gamma,\lambda,\theta,0)}$ in
\eqref{eq:Psi-for-delta=0} yields
\begin{equation}\label{eq:psi-theta=0-and-theta>0}
\Psi_{(\gamma,\lambda,\theta,0)}(x)=\theta+\Psi_{(\gamma,\lambda,0,0)}(x)-\lambda\theta^\gamma x\ (\mbox{for all}\ x\geq 0),
\end{equation}
where $\Psi_{(\gamma,\lambda,0,0)}$ is the function computed for the case $\theta=0$ (see
\eqref{eq:Psi-for-theta=delta=0}).

\subsection{The function $\Lambda_{(\gamma,\lambda,\theta,\delta)}$ in \eqref{eq:*}, for $\delta=0$}\label{sub:GW}
We restrict the attention to the case $\delta=0$ because we have an explicit expression
for $\kappa_{(\gamma,\lambda,\theta,\delta)}^*$ and, obviously, also for
$\Psi_{(\gamma,\lambda,\theta,\delta)}$. More precisely, by taking into 
account \eqref{eq:psi-theta=0-and-theta>0}, we consider
$\Psi_{(\gamma,\lambda,\theta,\delta)}$ in \eqref{eq:Psi-for-delta=0} for $\theta\geq 0$,
with $\gamma\in(0,1)$ when $\theta=0$.

So we have
$$\Lambda_{(\gamma,\lambda,\theta,0)}(y)=\sup_{x\geq 0}\{xy-\Psi_{(\gamma,\lambda,\theta,0)}(x)\}$$
where, if we consider the positive constant $c_\gamma$ defined by
$$c_\gamma:=\left\{\begin{array}{ll}
\gamma^{\gamma/(1-\gamma)}-\gamma^{1/(1-\gamma)}&\ \mbox{if}\ \gamma\in(0,1)\\
(-\gamma)^{\gamma/(1-\gamma)}+(-\gamma)^{1/(1-\gamma)}&\ \mbox{if}\ \gamma\in(-\infty,0),
\end{array}\right.$$
we have
$$\Psi_{(\gamma,\lambda,\theta,0)}(x)=\left\{\begin{array}{ll}
\theta-\lambda\theta^\gamma x+\lambda^{1/(1-\gamma)}c_\gamma x^{\gamma/(1-\gamma)}&\ \mbox{if}\ \gamma\in(0,1)\\
\theta+\lambda\theta^\gamma x-\lambda^{1/(1-\gamma)}c_\gamma x^{\gamma/(1-\gamma)}&\ \mbox{if}\ \gamma\in(-\infty,0)
\end{array}\right.\ \mbox{for all}\ x\geq 0.$$

Then we can state some results in the following lemma. Note that the next formula \eqref{eq:Lambda-gamma-positivo} 
with $\theta=0$ meets the expression of the limit in \eqref{eq:GE-limit-for-inverse-particular-case}.

\begin{lemma}\label{lem:GW}
We have:
\begin{equation}\label{eq:Lambda-gamma-positivo}
\Lambda_{(\gamma,\lambda,\theta,0)}(y)=\left\{\begin{array}{ll}
-\theta&\ \mbox{if}\ y<-\lambda\theta^\gamma\\
\left(\theta^\gamma+\frac{y}{\lambda}\right)^{1/\gamma}-\theta&\ \mbox{if}\ y\geq-\lambda\theta^\gamma
\end{array}\right.\ \mbox{for}\ \gamma\in(0,1),
\end{equation}
\begin{equation}\label{eq:Lambda-gamma-negativo}
\Lambda_{(\gamma,\lambda,\theta,0)}(y)=\left\{\begin{array}{ll}
\left(\theta^\gamma-\frac{y}{\lambda}\right)^{1/\gamma}-\theta&\ \mbox{if}\ y<\lambda\theta^\gamma\\
\infty&\ \mbox{if}\ y\geq\lambda\theta^\gamma
\end{array}\right.\ \mbox{for}\ \gamma\in(-\infty,0),
\end{equation}
and, in both cases,
$$\Psi_{(\gamma,\lambda,\theta,0)}(x)=\sup_{y\in\mathbb{R}}\{xy-\Lambda_{(\gamma,\lambda,\theta,0)}(y)\}.$$
\end{lemma}
\begin{proof}
All the results can be proved with some standard computations. The details are omitted.
\end{proof}

We conclude with another remark concerning both cases $\gamma\in(0,1)$ and $\gamma\in(-\infty,0)$.
Note that equation \eqref{eq:GW} in the following remark has some analogies with equations (12)-(13)
in \cite{GlynnWhitt} where the authors deal with counting processes, which are non-decreasing 
processes.

\begin{remark}\label{rem:GW}
For all $x\geq 0$ we have
$$x\kappa_{(\gamma,\lambda,\theta,0)}^*(1/x)=x\sup_{y\leq\theta}\{y/x-\kappa_{(\gamma,\lambda,\theta,0)}(y)\}
=\sup_{y\leq\theta}\{y-x\kappa_{(\gamma,\lambda,\theta,0)}(y)\};$$
moreover, if we consider the change of variable $z=-\kappa_{(\gamma,\lambda,\theta,0)}(y)$
and if we set
$$\mathcal{I}:=(-\kappa_{(\gamma,\lambda,\theta,0)}(\theta),-\kappa_{(\gamma,\lambda,\theta,0)}(-\infty)),$$
then we get
$$x\kappa_{(\gamma,\lambda,\theta,0)}^*(1/x)=
\sup_{z\in\mathcal{I}}\left\{\kappa_{(\gamma,\lambda,\theta,0)}^{-1}(-z)+xz\right\}=
\sup_{z\in\mathcal{I}}\left\{xz-(-\kappa_{(\gamma,\lambda,\theta,0)}^{-1}(-z))\right\}.$$
Thus $\Psi_{(\gamma,\lambda,\theta,0)}$ can be seen as the Fenchel-Legendre transform of
the function
\begin{equation}\label{eq:GW}
z\mapsto\tilde{\Psi}(z):=-\kappa_{(\gamma,\lambda,\theta,0)}^{-1}(-z),
\end{equation}
where $z$ belongs to a suitable set where the inverse function is well-defined.
In fact we have
$$\mathcal{I}=\left\{\begin{array}{ll}
(-\lambda\theta^\gamma,\infty)&\ \mbox{if}\ \gamma\in(0,1)\\
(-\infty,\lambda\theta^\gamma)&\ \mbox{if}\ \gamma\in(-\infty,0).
\end{array}\right.$$
and, in both cases $\gamma\in(-\infty,0)$ and $\gamma\in(0,1)$, the interval $\mathcal{I}$
coincides with the set where the function $\Lambda_{(\gamma,\lambda,\theta,0)}$ is strictly 
increasing and finite (see \eqref{eq:Lambda-gamma-positivo} and 
\eqref{eq:Lambda-gamma-negativo}).
\end{remark}

\section{Large deviations for time-changes with inverse processes}\label{sec:LD-time-changes}
The aim of this section is to present some applications of the G\"artner Ellis Theorem
in order to obtain LDPs for $\{X(T_{(\gamma,\lambda,\theta,\delta)}(t))/t:t>0\}$, when 
$\{X(t):t\geq 0\}$ is some suitable $\mathbb{R}^h$-valued process (for some integer 
$h\geq 1$), and independent of $\{T_{(\gamma,\lambda,\theta,\delta)}(t):t\geq 0\}$. Actually,
since we want to refer to the contents of Sections \ref{sub:theta=0} and \ref{sub:theta>0}, 
in this section we always restrict the attention to the case $\delta=0$. Moreover all the LDPs
stated in this section holds with speed $t$; therefore we always omit this detail (as we did 
in Section \ref{sec:LD-inverse-processes}).

The simplest case is when $\{X(t):t\geq 0\}$ is a L\'evy process; in fact we have
$$\mathbb{E}[e^{\langle\eta,X(t)\rangle}]=e^{t\Lambda_X(\eta)}\ (\mbox{for all}
\ \eta\in\mathbb{R}^h),\ \mbox{where}\ \Lambda_X(\eta):=\log\mathbb{E}[e^{\langle\eta,X(1)\rangle}].$$
In this case the application of the G\"artner Ellis Theorem works well when the function 
$\Lambda_X$ is finite in a neighborhood of the origin $\eta=0\in\mathbb{R}^h$; if $h=1$ this
means that all the random variables $\{X(t):t\geq 0\}$ are light tailed distributed (see e.g.
\cite{AsmussenAlbrecher}, Chapter I, Section 2).

A more general situation concerns additive functionals of Markov processes (here we recall
\cite{Veretennikov} as a reference with results based on the G\"artner Ellis Theorem); however,
for simplicity, we refer to the case of Markov additive processes (see e.g. 
\cite{AsmussenAlbrecher}, Chapter III, Section 4; actually the presentation in that reference 
concerns the case $h=1$). We have a Markov additive process $\{(J(t),X(t)):t\geq 0\}$ if, for 
some set $E$, it is a $E\times\mathbb{R}^h$-valued Markov process with suitable properties; in 
particular $\{J(t):t\geq 0\}$ is a Markov process. We refer to the continuous time case with a 
finite state space $E$ for $\{J(t):t\geq 0\}$; see e.g. \cite{AsmussenAlbrecher}, page 55. We 
also assume that $\{J(t):t\geq 0\}$ is irreducible and, for simplicity, that
$\mathbb{E}[e^{\langle\eta,X(t)\rangle}]<\infty$ for all $\eta\in\mathbb{R}^h$.
Then, as a consequence of Proposition 4.4 in Chapter III in \cite{AsmussenAlbrecher}, we have
$$\min_{i\in E}h_i(\eta)e^{t\Lambda_X(\eta)}\leq\mathbb{E}[e^{\langle\eta,X(t)\rangle}]\leq\max_{i\in E}h_i(\eta)e^{t\Lambda_X(\eta)}$$
where $e^{t\Lambda_X(\eta)}$ is a suitable simple and positive eigenvalue and 
$(h_i(\eta))_{i\in E}$ is a positive eigenvector (these items can be found by a suitable 
application of the Perron Frobenius Theorem).

Now we are ready to illustrate the applications of the G\"artner Ellis Theorem which 
provides the LDP for $\{X(T_{(\gamma,\lambda,\theta,\delta)}(t))/t:t>0\}$ with rate function 
$H_{(\gamma,\lambda,\theta,0)}$, say. In particular we can have a trapping and delaying
effect for $\theta=0$ (see Remark \ref{rem:trapping-delaying}), and a possible rushing 
effect for $\theta>0$; we recall a recent reference with this kind of analysis for 
time-changed processes is \cite{CapitanelliDovidio}, even if the approach in this paper
is different from the one in that reference. We also give some comments on the behavior of 
$H_{(\gamma,\lambda,\theta,0)}(x)$ around the origin $x=0$ for $h=1$; this will be done for
both cases $\theta=0$ and $\theta>0$, and we see that right and left derivatives at $x=0$ 
(which will be denoted by $D_-H_{(\gamma,\lambda,\theta,0)}(0)$ and 
$D_+H_{(\gamma,\lambda,\theta,0)}(0)$) can be different.

\subsection{Case $\theta=0$}\label{sub:theta=0-tc}
Here we refer to the content of Section \ref{sub:theta=0}. We also recall that we only have 
$\gamma\in(0,1)$ when $\theta=0$. Then, after some standard computations (with a conditional 
expectation with respect to the independent random time-change), we get
\begin{multline*}
\lim_{t\to\infty}\frac{1}{t}\log\mathbb{E}[e^{\langle\eta,X(T_{(\gamma,\lambda,0,0)}(t))\rangle}]=
\lim_{t\to\infty}\frac{1}{t}\log E_\gamma\left(\frac{\Lambda_X(\eta)}{\lambda}t^\gamma\right)\\
=\left\{\begin{array}{ll}
(\Lambda_X(\eta)/\lambda)^{1/\gamma}&\ \mbox{if}\ \Lambda_X(\eta)\geq 0\\
0&\ \mbox{if}\ \Lambda_X(\eta)<0
\end{array}\right.=\Lambda_{(\gamma,\lambda,0,0)}(\Lambda_X(\eta)),
\end{multline*}
where $\Lambda_{(\gamma,\lambda,0,0)}(\cdot)$ is the function in \eqref{eq:GE-limit-for-inverse-particular-case}
(see also \eqref{eq:Lambda-gamma-positivo} with $\theta=0$).

Then, under suitable hypotheses, by the G\"artner Ellis Theorem, $\{X(T_{(\gamma,\lambda,0,0)}(t))/t:t>0\}$
satisfies the LDP with good rate function $H_{(\gamma,\lambda,0,0)}$ defined by
$$H_{(\gamma,\lambda,0,0)}(x):=\sup_{\eta\in\mathbb{R}^h}\{\langle\eta,x\rangle-\Lambda_{(\gamma,\lambda,0,0)}(\Lambda_X(\eta))\}.$$
We can say that $H_{(\gamma,\lambda,0,0)}(x)=0$ if and only if $x=\Lambda_{(\gamma,\lambda,0,0)}^\prime(\Lambda_X(0))\nabla\Lambda_X(0)$;
thus, since $\Lambda_X(0)=0$ and $\Lambda_{(\gamma,\lambda,0,0)}^\prime(0)=0$, 
we have $H_{(\gamma,1,0,0)}(x)=0$ if and only if $x=0$, whatever is 
$\nabla\Lambda_X(0)$. 

\begin{remark}\label{rem:trapping-delaying}
We can say that $\frac{X(T_{(\gamma,\lambda,0,0)}(t))}{t}$ converges to zero as $t\to\infty$
(at least in probability; see Remark \ref{rem:convergence-under-GET} for a discussion on the 
almost sure of $\frac{X(T_{(\gamma,\lambda,0,0)}(t_n))}{t_n}$ along a sequence 
$\{t_n:n\geq 1\}$ such that $t_n\to\infty$), and this happens whatever is the limit 
$\nabla\Lambda_X(0)$ of $\frac{X(t)}{t}$. This is not surprising because random 
time-changes with $\{T_{(\gamma,\lambda,0,0)}(t):t\geq 0\}$ typically give rise to a sort of 
trapping and delaying effect; a discussion on this aspect for random time-changes can be
found in \cite{CapitanelliDovidio}.
\end{remark}

We conclude with some statements for the case $h=1$. In what follows we consider certain inequalities; 
however similar statements hold if we consider inverse inequalities. We assume that $\Lambda_X^\prime(0)>0$.
\begin{itemize}
\item If there exists $\eta_0<0$ such that $\Lambda_X(\eta_0)=0$ (note that this condition can occur because
$\Lambda_X$ is convex and $\Lambda_X(0)=0$), we can say that $D_-H_{(\gamma,\lambda,0,0)}(0)=\eta_0$ and 
$D_+H_{(\gamma,\lambda,0,0)}(0)=0$.
\item On the contrary, if $\Lambda_X$ is strictly increasing (and therefore uniquely vanishes at $\eta=0$), we have 
$H_{(\gamma,\lambda,0,0)}(x)=\infty$ for all $x<0$ instead of $D_-H_{(\gamma,\lambda,0,0)}(0)=\eta_0$.
\end{itemize}

%WHAT HAPPENS WITH INVERSE INEQUALITIES.
%Assume that $\Lambda_X^\prime(0)<0$. Then, if there exists $\eta_0>0$
%such that $\Lambda_X(\eta_0)=0$ (note that this condition can occur because $\Lambda_X$ is convex and 
%$\Lambda_X(0)=0$), we can say that $D_-H_{(\gamma,\lambda,0,0)}(0)=0$ and $D_+H_{(\gamma,\lambda,0,0)}(0)=\eta_0$.
%On the contrary, if $\Lambda_X$ is strictly decreasing (and therefore uniquely vanishes at $\eta=0$), we have 
%$H_{(\gamma,\lambda,0,0)}(x)=\infty$ for all $x>0$ instead of $D_+H_{(\gamma,\lambda,0,0)}(0)=\eta_0$.

\subsection{Case $\theta>0$}\label{sub:theta>0-tc}
Here we refer to the content of Section \ref{sub:theta>0}. We start with the same standard computations 
considered in Section \ref{sub:theta=0} but here we cannot refer to \eqref{eq:GE-limit-for-inverse-particular-case}.
In fact in this case we refer to Claim \ref{claim:DW2} in order to have the limit \eqref{eq:GE-limit-naive-approach}
for all $y\in\mathbb{R}$; so, as stated in Claim \ref{claim:DW2}, we take $\gamma\in(0,1)$ in order to 
have $\kappa_{(\gamma,\lambda,\theta,\delta)}^*(0)=\infty$. Then we get
$$\lim_{t\to\infty}\frac{1}{t}\log\mathbb{E}[e^{\langle\eta,X(T_{(\gamma,\lambda,\theta,0)}(t))\rangle}]
=\Lambda_{(\gamma,\lambda,\theta,0)}(\Lambda_X(\eta));$$
moreover $\Lambda_{(\gamma,\lambda,\theta,0)}(\cdot)$ is given by \eqref{eq:Lambda-gamma-positivo}.

Then, under suitable hypotheses, by the G\"artner Ellis Theorem,
$\{X(T_{(\gamma,\lambda,\theta,0)}(t))/t:t>0\}$ satisfies the LDP with good rate function
$H_{(\gamma,\lambda,\theta,0)}$ defined by
$$H_{(\gamma,\lambda,\theta,0)}(x):=\sup_{\eta\in\mathbb{R}^h}\{\langle\eta,x\rangle-\Lambda_{(\gamma,\lambda,\theta,0)}(\Lambda_X(\eta))\}.$$
We can say that $H_{(\gamma,\lambda,\theta,0)}(x)=0$ if and only if $x=\Lambda_{(\gamma,\lambda,\theta,0)}^\prime(\Lambda_X(0))\nabla\Lambda_X(0)$;
thus, since $\Lambda_X(0)=0$ and $\Lambda_{(\gamma,\lambda,\theta,0)}^\prime(0)=\frac{\theta^{1-\gamma}}{\lambda\gamma}$ (note that
$\Lambda_{(\gamma,\lambda,\theta,0)}^\prime(0)=(\kappa_{(\gamma,\lambda,\theta,0)}^\prime(0))^{-1}$ as one can expect), we have 
$H_{(\gamma,\lambda,\theta,0)}(x)=0$ if and only if $x=\frac{\theta^{1-\gamma}}{\lambda\gamma}\nabla\Lambda_X(0)$.
So $X(T_{(\gamma,\lambda,\theta,0)}(t))/t$ converges to a limit that depends on $\nabla\Lambda_X(0)$, and we have a possible rushing effect.

We conclude with some statements for the case $h=1$. In what follows we consider certain inequalities; 
however similar statements hold if we consider inverse inequalities.
\begin{itemize}
\item If there exists $\eta_1<\eta_2<0$ such that $\Lambda_X(\eta_1)=\Lambda_X(\eta_2)=-\lambda\theta^\gamma$
(and this happens if $\Lambda_X^\prime(0)>0$), then $D_-H_{(\gamma,\lambda,\theta,0)}(0)=\eta_1$ and 
$D_+H_{(\gamma,\lambda,\theta,0)}(0)=\eta_2$.
\item On the contrary, if there exists a unique $\eta_0<0$ such that $\Lambda_X(\eta_0)=-\lambda\theta^\gamma$
(and this could happen if $\Lambda_X$ is strictly increasing) we have $D_+H_{(\gamma,\lambda,\theta,0)}(0)=\eta_0$
and $H_{(\gamma,\lambda,\theta,0)}(x)=\infty$ for $x<0$.
\end{itemize}

\begin{remark}\label{rem:HversusPsi}
Note that $H_{(\gamma,\lambda,\theta,0)}$ coincides with $\Psi_{(\gamma,\lambda,\theta,0)}$ when we have
$X(t)=t$ for all $t\geq 0$. In such a case $\Lambda_X(\eta)=\eta$ for all $\eta\in\mathbb{R}$ and 
therefore we have $\Lambda_X(\eta_0)=-\lambda\theta^\gamma$ for $\eta_0=-\lambda\theta^\gamma<0$. Thus we get $D_+H_{(\gamma,\lambda,\theta,0)}(0)=-\lambda\theta^\gamma$ and this agrees with the right derivative of
$\Psi_{(\gamma,\lambda,\theta,0)}(y)$ at $y=0$ in \eqref{eq:right-derivative-of-Psi-at-zero} for 
$\gamma\in(0,1)$.
\end{remark}

%WHAT HAPPENS WITH INVERSE INEQUALITIES.
%If there exists $\eta_2>\eta_1>0$ such that $\Lambda_X(\eta_1)=\Lambda_X(\eta_2)=-\lambda\theta^\gamma$
%(and this happens if $\Lambda_X^\prime(0)<0$), then $D_-H_{(\gamma,\lambda,\theta,0)}(0)=\eta_1$ and
%$D_+H_{(\gamma,\lambda,\theta,0)}(0)=\eta_2$. On the contrary, if there exists a unique $\eta_0>0$ such that
%$\Lambda_X(\eta_0)=-\lambda\theta^\gamma$ (and this could happen if $\Lambda_X$ is strictly decreasing), we
%have $D_-H_{(\gamma,\lambda,\theta,0)}(0)=\eta_0$ and $H_{(\gamma,\lambda,\theta,0)}(x)=\infty$ for $x>0$.

\subsection*{Acknowledgements}
We thank an anonymous referee for the careful reading of an earlier version of the manuscript,
and for some useful comments. We also thank Luisa Beghin for some minor comments on the 
presentation.

\end{document}